\documentclass[a4paper,12pt,final]{amsart}
\usepackage{times,a4wide,mathrsfs}

\usepackage{amsmath}
\usepackage{amsfonts}

\newcommand{\RR}{\mathbb{R}}
\newcommand{\QQ}{\mathbb{Q}}

\newcommand{\amp}{\hbox{Amp}}
\newcommand{\nef}{\hbox{Nef}}

\newcommand{\psef}{\hbox{Psef}}
\newcommand{\bigc}{\hbox{Big}}
\newcommand{\mob}{\hbox{Mob}}
\newcommand{\vis}{\partial\nef(X)_{\rm visible}}

\newcommand{\inte}{\hbox{int}}

\newcommand{\exc}{\hbox{Exc}}
\newcommand{\wt}{\widetilde}

\newcommand{\rom}{\romannumeral}

\newtheorem{theorem}{Theorem}

\newtheorem{lemma}[theorem]{Lemma}
\newtheorem{corollary}[theorem]{Corollary}
\newtheorem{proposition}[theorem]{Proposition}

\newtheorem{remark}[theorem]{Remark}
\newtheorem{definition}[theorem]{Definition}
\newtheorem{convention}{Convention}

\newtheorem{nonumbering}{Theorem}

\newtheorem{nonumberingc}{Corollary}

\begin{document}
\author[Robert Laterveer]
{Robert Laterveer}

\address{Institut de Recherche Math\'ematique Avanc\'ee, Universit\'e 
de Strasbourg, 7 Rue Ren\'e Des\-car\-tes, 67084 Strasbourg CEDEX, France.}
\email{robert.laterveer@math.unistra.fr}

\date{\today}



\title{Some remarks on cones of partially ample divisors}

\begin{abstract} We study the cones of $q$-ample divisors $q\amp$ on smooth complex varieties. In favourable cases, we identify a part where the closure $\overline{q\amp}$ and the nef cone have the same boundary. This is especially interesting for Fano (or almost Fano) varieties.
\end{abstract}
\vskip 1cm

\maketitle

Totaro's landmark paper \cite{Tot}  has given a new impetus to the study of partially ample divisors. Let $X$ be a smooth projective complex variety of dimension $n$, and $L$ on $X$ a line bundle.  We recall that $L$ is called $q$-ample if for every coherent sheaf ${\mathcal F}$ there exists an integer $m_0$ such that
  \[ H^i(X,{\mathcal F}\otimes L^{\otimes m}) \hbox{ for all $i>q$ and $m>m_0$.}\]
From Serre's criterion it follows that $0$-ampleness coincides with ampleness.  
Totaro proves that the $q$-ampleness of $L$ only depends on the numerical equivalence class of $L$ \cite[Theorem 8.3]{Tot}. The definition can moreover be extended to $\RR$-divisors \cite[8.2]{Tot}, in such a way that
$q$-ample $\RR$-divisors form an open cone $q\amp(X)$ in $N^1(X)$ (the space of $\RR$-divisors modulo numerical equivalence). We thus get a series of cones
\[  \amp(X)=0\amp(X)\subset 1\amp(X)\subset\cdots   \subset n\amp(X)=N^1(X)\ .\]
While the ample cone $\amp(X)$ and the cone $(n-1)\amp(X)$ are fairly well understood, the intermediate cones $q\amp(X)$ for $0<q<n-1$ are still quite elusive and mysterious (see for instance \cite[section 11]{Tot} for some fundamental open questions).

The modest goal of this paper is to identify a part of these cones $q\amp$. Indeed, it turns out that in favourable cases, part of the boundary of the closed cone $\overline{q\amp}$ coincides with the boundary of the nef cone. To start with, let's restrict attention to the case that is easiest to state, that of the cone of $1$-ample divisors $1\amp$. Let $\partial\nef(X)$ denote the boundary of the nef cone, and let $K_X\in N^1X$ denote the class of the canonical divisor.
We define
  \[  \partial\nef(X)_{\rm visible}\subset \partial\nef(X)\]
  to be the part of the boundary that is visible from $K_X$; cf. Definition \ref{visible} for the precise definition. (We note that when $K_X$ is nef, we have $\partial\nef(X)_{\rm visible}=\emptyset$ !)
  
This ``$K_X$--visible part'' of the boundary turns out to be closely related to the boundary of $1\amp(X)$. This is detailed in the following result, where $\mob(X)$ and $\bigc(X)$ denote the cone of mobile divisors resp. big divisors.
  
\begin{nonumbering}{(=Theorem \ref{1amp})} Let $X$ be a smooth projective complex variety. 

\item{(\rom1)} 
\[\partial\nef(X)_{\rm visible}\cap\inte\bigl(\mob(X)\bigr)\] is in the boundary of $\overline{1\amp(X)}$.

\item{(\rom2)} Suppose $X$ is not the blow--up of a smooth projective variety along a smooth codimension $2$ subvariety. Then
  \[  \vis\cap\bigc(X) \subset \partial\overline{1\amp(X)}\ .\]
  
\item{(\rom3)} Suppose $X$ is not a conic bundle over a smooth projective variety, nor a blow--up of a smooth projective variety along a smooth codimension $2$ subvariety. Then
  \[  \vis \subset \partial\overline{1\amp(X)}\ .\]
\end{nonumbering}

That is, with two exceptions (a blow--up and a conic bundle) the ample cone and the $1$-ample cone look exactly the same when observed from $K_X$, and hence the only places where $1\amp$ can grow larger than $\amp$ are located in the ``shadowy part'' invisible from $K_X$. This theorem is proven by exploiting the existence of an MMP for any adjoint divisor, as proven by Birkar--Cascini--Hacon--McKernan \cite{BCHM}.

It follows from Theorem \ref{1amp} that the cone $1\amp(X)$ is strictly convex for any $X$ such that $\vis\cap\inte\bigl(\mob(X)\bigr)\not=\emptyset$ (Corollary \ref{convex}).
The following is also an immediate corollary:

\begin{nonumberingc}{(=Corollary \ref{kx1amp})} Let $X$ be a smooth projective variety, and suppose $K_X$ is $1$-ample. Then
  \[ \vis\subset \partial\overline{\mob(X)}\ .\]
\end{nonumberingc}

That is, if $K_X$ is $1$-ample the nef cone and the closed mobile cone look the same when observed from $K_X$.

 Of course, the above theorem is empty of content when $K_X$ is nef (for then the $K_X$--visible part is empty), while the assertion grows stronger when $K_X$ grows more negative (for then the $K_X$--visible part grows larger, which means that the $1$-ample cone looks more and more like the ample cone). The limit case is when $X$ is a Fano variety: then the whole boundary of $\nef(X)$ is $K_X$--visible. In fact, we can prove more generally:

\begin{nonumberingc}{(=Corollary \ref{fano})} Let $X$ be a smooth projective complex variety such that either (1) $-K_X$ is ample, or (2) $-K_X$ is $\not=0$ and nef and $\dim N^1X\ge 3$. Then:

\item{(\rom1)} 
\[\partial\nef(X)\cap\inte\bigl(\mob(X)\bigr)\] is in the boundary of $\overline{1\amp(X)}$.

\item{(\rom2)} Suppose $X$ is not the blow--up of a smooth projective variety $Y$ along a smooth codimension $2$ subvariety. Then
  \[  \partial\nef(X)\cap\bigc(X) \subset \partial\overline{1\amp(X)}\ .\]
  
\item{(\rom3)} Suppose $X$ is not a conic bundle over a smooth projective variety $Y$, nor a blow--up of a smooth projective variety along a smooth codimension $2$ subvariety. Then
  \[  \amp(X)=1\amp(X)\ .\]
\end{nonumberingc}   

(For Fano varieties, I proved this in \cite{Lat}).

Here is an application of the above theorem: we can identify a part of the nef cone for which the weak Lefschetz principle holds. Let $Y\subset X$ be a generic hyperplane section. If the dimension $n$ of $X$ is $\ge 4$, pull-back induces a natural isomorphism $N^1X\cong N^1Y$. Thus it makes sense to ask whether the nef cones $\nef(X)$ and $\nef(Y)$ coincide. The answer is negative in general, as shown by Hassett--Lin--Wang \cite{Has}. On the other hand, the answer is positive for certain Fano varieties (\cite{Wis}, \cite{Has}, \cite{Jow}, \cite{ANO}, \cite{BI}, \cite{Sz}). Using the above Theorem, it turns out that the $K_X$--visible part cuts out a part where weak Lefschetz holds for the nef cone:

\begin{nonumberingc}{(=Corollary \ref{weak})} Let $X$ be a smooth projective complex variety of dimension $n\ge 4$, and let $Y\subset X$ be any ample hypersurface. Then
\[\vis\cap\inte\bigl(\mob(X)\bigr)\subset  \partial\nef(Y)\cap\partial\nef(X)\ .\]

\end{nonumberingc}   

This is proven using a result of Demailly--Peternell--Schneider \cite{DPS} (cf. also \cite{Kur}), which says that a divisor restricting to an ample divisor on $Y$ is $1$-ample on $X$.

We prove a result similar to Theorem \ref{1amp}, by similar means, for the $q$-ample cone (where $q$ may be $>1$). This result is a bit more awkward to state. As a matter of notation, we introduce the cone $Bq\amp(X)$; this is defined as the cone of those $\RR$-divisors which have augmented base locus of dimension $\le q$.

\begin{nonumbering}{(=Theorem \ref{qample})} Let $X$ be a smooth projective variety of dimension $n$.  For any non--negative integer $q$, we have

\[   \vis\cap B(n-1-q)\amp(X)\subset \partial\overline{q\amp(X)}\ .\]
\end{nonumbering}
 
Here is how this paper is organized. The first two sections are of a preliminary nature. The first concerns several cones of divisors related to the $q$-ample cones; the second contains some results about contractions that will be needed. Section $3$ contains the proof of Theorem \ref{1amp} and its corollaries. In section $4$, we prove Theorem \ref{qample}. 

Helpful conversations with Gianluca Pacienza are gratefully acknowledged.

\begin{convention} In this paper, all varieties will be (quasi--)projective algebraic varieties defined over the complex numbers.
\end{convention}

\section{Cones}

This section contains notation and basic results concerning several cones of divisors related to the $q$-ample cones. These cones have been introduced by K\"uronya \cite{Kur} and de Fernex--K\"uronya--Lazarsfeld \cite{FKL}.

\begin{definition} Let $X$ be a projective variety.
A line bundle $L$ on $X$ is called $q$-ample if for every coherent sheaf ${\mathcal F}$ there exists an integer $m_0$ such that
  \[ H^i(X,{\mathcal F}\otimes L^{\otimes m}) \hbox{ for all $i>q$ and $m>m_0$.}\]
A $\QQ$--Cartier divisor is called $q$-ample if some integral multiple is $q$-ample. An $\RR$--Cartier divisor $D$ is called $q$-ample if it can be written as a sum 
  \[ D=cL+A\ ,\]
where $c\in\RR_{>0}$, $L$ is a $q$-ample line bundle and $A$ is an ample $\RR$--Cartier divisor. We will denote
  \[  q\amp(X)\subset N^1(X)\]
  the cone generated by $q$-ample divisors.
\end{definition}

\begin{remark} The consistency of the definition for $\RR$--divisors with the one for $\QQ$--divisors is proven by Totaro
\cite[Theorem 8.3]{Tot}. The cones $q\amp(X)$ are open cones \cite[Theorem 8.3]{Tot}.
\end{remark}  

\begin{theorem}{(\cite[Theorem 9.1]{Tot})}\label{n-1} Let $X$ be a projective variety of dimension $n$. The cone $(n-1)\amp(X)$ is the complement in $N^1X$ of the negative of the pseudo--effective cone of $X$. 
\end{theorem}

\begin{definition} Let $X$ be a projective variety.
\item{(\romannumeral 1)}
An $\RR$--divisor $L$ on $X$ is called B $q$-ample if the augmented base locus $B_+(L)$ has dimension $\le q$.
We will denote
  \[ Bq\amp(X)\subset N^1(X)\]
 the cone generated by B $q$-ample divisors. 

\item{(\romannumeral 2)} Let $H_1,\ldots,H_q$ be very ample divisors on $X$. An $\RR$--divisor $L$ on $X$ is called 
$(H_1,\ldots,H_q)$-ample if the restriction 
  \[L\vert_{h_1\cap\cdots\cap h_q}\]
   is ample, for $h_i\in\vert H_i\vert$ generic. An $\RR$--divisor is said to be
H $q$-ample if it is $(H_1,\ldots,H_q)$-ample, for certain very ample $H_1, \ldots, H_q$. We will denote
  \[ Hq\amp(X)=\bigcup_{(H_1,\ldots,H_q)\hbox{very ample}} (H_1,\ldots,H_q)\amp(X)\subset N^1(X)\]
  the cone generated by H $q$-ample divisors.
\end{definition}

\begin{remark} The augmented base locus $B_+(L)\subset X$ is the locus where $L$ fails to be ample; for the definition and properties, cf. \cite{ELM} and \cite{ELM2}.
\end{remark}

\begin{remark}\label{mobile} It is easily seen that
  \[ B0\amp(X)=H0\amp(X)=\amp(X),\]
  while $B(n-1)\amp(X)=\hbox{Big}(X)$. The cones $Bq\amp(X)$ are open \cite[Theorem 4.5]{Choi},
     and $B(n-2)\amp(X)$ coincides with the interior of the cone of mobile divisors:
   \[  B(n-2)\amp(X)=  \mob(X)  \setminus \partial\mob(X) \] 
         (\cite[Lemma 3.1]{Choi2}). 
         \end{remark}

\begin{remark}\label{Cho} The cones $Bq\amp$ (or rather, their closure) have been studied by Payne \cite{Pay} and Choi \cite{Choi}. It is established by Choi \cite[Theorem 4.5]{Choi} that the closure of $Bq\amp(X)$ can be described in terms of the diminished base locus:
  \[    \overline{Bq\amp(X)}=\{ L\in N^1X \vert \dim B_-(L)\le q\}\ . \]  
\end{remark}

\begin{proposition}\label{inclusions}(K\"uronya \cite{Kur}) Let $X$ be a smooth projective variety. For any $0\le q\le n-1$, there are inclusions of cones
  \[ Bq\amp(X)\subset Hq\amp(X)\subset q\amp(X).\]
\end{proposition}  

\begin{proof} For the first inclusion, it is easily seen that actually
  \[ Bq\amp(X)\subset \bigcap_{(H_1,\ldots,H_q)\hbox{very ample}} (H_1,\ldots,H_q)\amp(X);\]
 indeed, suppose $L$ is such that $\dim B_+(L)\le q$. For any $H_1,\ldots,H_q$ very ample and $h_i\in\vert H_i\vert$ generic,
 $B_+(L)\cap h_1\cap\cdots h_q$ has dimension $\le 0$. But 
  \[ B_+(L\vert_{   h_1\cap\cdots\cap h_q})\subset B_+(L)\cap h_1\cap\cdots h_q\]
\cite[]{Kur}, so $L  \vert_{   h_1\cap\cdots\cap h_q}$ is ample. The second inclusion is a vanishing theorem proven by K\"uronya \cite[Theorem 1.1]{Kur}; this was also proven by Demailly--Peternell--Schneider \cite[Theorem 3.4]{DPS}.   
\end{proof}

\begin{remark} Both inclusions in Proposition \ref{inclusions} may be strict. For the second inclusion, K\"uronya provides an example \cite[Example 1.13]{Kur} where
  \[ H(n-1)\amp(X)  \not= (n-1)\amp(X). \]
 For the first inclusion, let $X$ be a surface. Then any line bundle $L$ which is not big and such that $-L$ is not pseudo-effective is in
 \[ H1\amp(X)\setminus B1\amp(X).\]
 A more subtle example is \cite[Example 1.7]{Kur}, which exhibits a big line bundle $L$ on a threefold $X$, satisfying
 \[ L\in H1\amp(X)\setminus B1\amp(X).\]
 \end{remark}

 \section{MMP}
 
 In this section, we collect some results about minimal model theory and contractions.
 
 \begin{definition}\label{stable}{(\cite{ELM2})} A divisor $L$ is called {\sl stable} if $B_-(L)$ and $B_+(L)$ coincide.
\end{definition}

\begin{proposition}\label{dense}{(\cite[Proposition 1.29]{ELM2})} The stable divisors form an open and dense subset in $N^1X$.
\end{proposition}

\begin{lemma}\label{indet} Let $X$ be a smooth projective variety, and $L$ on $X$ an $\RR$--divisor which is big and stable.
 Let 
 \[f\colon X - - \to X_{min}\]
  be an $L$--MMP, i.e. $f_\ast L$ is nef. Let $\exc(f)\subset X$ denote the complement of the maximal open subset over which $f$ is an isomorphism. Then
   \[  B_+(L)\supset\exc(f)\ .\]
 \end{lemma}  
 
 \begin{proof}
Let $E\subset \exc(f)$ be an irreducible component. Then, there is some index $0<i<r$, such that $-(f_i)_\ast L$ is $\psi_i$--ample on the strict transform $E_i$ of $E$ in $X_i$. This implies
   \[  E_i\subset  B_+\bigl((f_i)_\ast L\bigr)\]
 (indeed, $E_i$ is covered by curves on which $(f_i)_\ast L$ is negative, and such curves lie in the stable base locus of $(f_i)_\ast L$).  But then, applying the following proposition to a resolution of indeterminacy of $f_i$, we see that $E$ must lie in $B_+(L)$.
 
\begin{proposition}\label{bbp}{(Boucksom--Broustet--Pacienza \cite[Proposition 1.5]{BBP})} Let $\pi\colon \wt{X}\to X$ be a birational morphism between normal projective varieties. Let $F$ be an effective $\pi$--exceptional divisor. Then for any big $\RR$--divisor $L$ on $X$, we have
  \[  B_+(\pi^\ast L+F)=\pi^{-1}\bigl( B_+(L)\bigr)\cup \exc(\pi)\ .\]
\end{proposition}  

\end{proof}

\begin{remark} With some more work, one can in fact prove that equality holds in Lemma \ref{indet}; we don't need this in this paper.
\end{remark}

\begin{theorem}\label{dim1}{(Wi\'sniewski \cite{Wis})} Let $X$ be a smooth projective variety, and let
  \[ \psi\colon X\to Z\]
be the contraction of a $K_X$--negative extremal ray. Suppose all fibres of $\psi$ are of dimension $\le 1$. Then $Z$ is smooth, and $\psi$ is either the blow--up of $Z$ along a smooth codimension $2$ subvariety, or a conic bundle over $Z$.
\end{theorem}

\begin{proof} \cite[Theorem 1.2]{Wis} (cf. also \cite[Theorem 4.1]{AW}.
\end{proof}

 \begin{theorem}\label{wis}{(Wi\'sniewski \cite{Wis}, Ionescu \cite{Ion})} Let $X$ be a smooth projective variety of dimension $n$, and let $R$ be a $K_X$--negative extremal ray of length 
   \[ \ell(R):=\min\bigl\{ -K_X\cdot C \vert C\hbox{\ rational\ curve}, C\in R\bigr\}\ .\]
Let $\psi$ be the contraction of $R$, and let $E$ be an irreducible component of the locus of $R$. Let $F$ be an irreducible component of a fiber of the restriction of $\psi$ to $E$. Then
  \[ \dim E+\dim F \ge n+\ell(R)-1\ .\]
\end{theorem}

\begin{proof} \cite[Theorem 1.1]{Wis} or \cite[Theorem 0.4]{Ion}. 
%or \cite[IV Corollary 2.6]{Kol}
\end{proof}

\section{$1$-ample}

This section is about the cone of $1$-ample divisors. Here we prove Theorem \ref{1amp} stated in the introduction.

\begin{definition}\label{visible} Let $X$ be a projective variety. The $K_X$--visible part of $\partial\nef(X)$
is defined as
  \[  \vis:=\{ D\in\partial\nef(X)\vert \hbox{\ \ }\overline{K_X D}\cap\nef(X)=D\}\ .\]
Here $\overline{K_X D}$ denotes the line segment joining $K_X$ to $D$.
\end{definition}

\begin{remark} This notion is considered also in \cite[Theorem 1]{Kaw}.  The definition is interesting only when $K_X\not\in\nef(X)$; if $K_X$ is nef, the line segment $\overline{K_X D}$ contains more than one point and we have
  \[ \vis=\emptyset\ .\]
  The other extreme is when $X$ is Fano; then we have
  \[  \vis=\partial\nef(X)\ .\]
\end{remark}

\begin{theorem}\label{1amp} Let $X$ be a smooth projective variety. 

\item{(\rom1)} 
\[\partial\nef(X)_{\rm visible}\cap\inte\bigl(\mob(X)\bigr)\subset   \partial\overline{1\amp(X)}\ .\]

\item{(\rom2)} Suppose $X$ is not the blow--up of a smooth projective variety $Y$ along a smooth codimension $2$ subvariety. Then
  \[  \vis\cap\bigc(X) \subset \partial\overline{1\amp(X)}\ .\]
  
\item{(\rom3)} Suppose $X$ is not a conic bundle over a smooth projective variety $Y$, nor a blow--up of a smooth projective variety along a smooth codimension $2$ subvariety. Then
  \[  \vis \subset \partial\overline{1\amp(X)}\ .\]
\end{theorem}

\begin{proof} 

\item{(\rom1)} We will prove the following:

\begin{proposition}\label{aux} Let $L=K_X+A$, where $A$ is an ample $\RR$-divisor. Suppose $L$ is stable and
  \[ L\in 1\amp(X)\cap\inte\bigl( \mob(X)\bigr)\ .\]
Then $L$ is ample.
\end{proposition}  

This suffices to prove Theorem \ref{1amp}(\rom1). Indeed, 
suppose there is an element 
  \[ D\in\vis\cap\inte\bigl( \mob(X)\bigr)  \]
 that is in the interior of $\overline{1\amp(X)}$ (i.e. $D$ is $1$-ample). 
 Then we can also find 
   \[ D^\prime\in\vis^{\circ}\cap\inte\bigl( \mob(X)\bigr)  \]
 that is $1$-ample. Here $\vis^{\circ}$ denotes the relative interior of $\vis$. 
  By definition of the $K_X$--visible part, $D^\prime$ is of the form $D=m(K_X+A)$, for some ample $\RR$-divisor $A$ and $m\in\RR$. Now, ${1\over m}D^\prime=K_X+A$ is also
 in
 \[   \vis\cap\inte\bigl( \mob(X)\bigr)\cap 1\amp(X)\ .  \]
What's more, 
  \[D^{\prime\prime}=K_X+(1-\epsilon)A \in\inte\bigl(\mob(X)\bigr)\cap 1\amp(X)\] 
for $0<\epsilon$ small enough (since these are open cones). Since stable divisors are open and dense in $N^1X$, there exists $\epsilon>0$ such that $D^{\prime\prime}$ is stable. Then Proposition \ref{aux} implies that $D^{\prime\prime}$ is ample, and hence $D^\prime$ is ample: contradiction.

So let's prove Proposition \ref{aux}. 

%\begin{lemma} Let $X$ be a smooth projective variety, and $A$ an ample $\RR$-divisor. There exists an effective $\RR$-%divisor $\Delta$, numerically equivalent to $A$, such that the pair $(X,\Delta)$ is klt.
%\end{lemma}

%\begin{proof} Suppose $A=\sum_i a_i A_i$, where $a_i\in\RR$ and the $A_i$ are very ample line bundles. For any $m\in%\NN$, let $A_i^m$ be a generic section of the linear system $\vert mA_i\vert$. By Bertini, the $A_i^m$ are smooth, and
 % \[ A_i={1\over m} A_i^m\in N^1X\ .\]
 % For $m$ large enough, we thus get
 % \[  A=\Delta=\sum a_i^\prime A_i^\prime\in N^1X\ ,\]
 % with $a_i^\prime<1$ and $\sum A_i^\prime$ a normal crossing divisor. This implies that $(X,\Delta)$ is klt.
%\end{proof}  

Since $A$ is ample, there exists an effective $\RR$-divisor $\Delta$ numerically equivalent to $A$ and such that $(X,\Delta)$ is klt. According to \cite[Theorem 1.2]{BCHM}, there is an $L$--MMP
  \[ \phi\colon X=X_0 -  \to X_1 - \to \cdots  - \to X_{min},\]
where $\phi_\ast L$ on $X_{min}$ is nef. Each step $\phi_i\colon X_i  -\to X_{i+1}$ in the program is the flip of a morphism 
  \[\psi_i\colon X_i\to Z_i\ ,\]
where $\psi_i$ is the (birational) contraction of an $L$--negative extremal ray. Since $L$ is stable, the exceptional locus of $\phi$ is contained in $B_+(L)$ (Lemma \ref{indet}), hence it is of dimension $\le n-2$ (where $n=\dim X$).  That is, all the $\psi_i$ in the program must be small contractions. Consider now the first of these small contractions
  \[ \psi=\psi_0\colon X\to Z_0\ .\]
Since $K_X<L$, $\psi$ is the contraction of a $K_X$--negative extremal ray.  
If all fibres of $\psi$ are of dimension $\le 1$, the contraction $\psi$ cannot be small by Theorem \ref{dim1}, so there must exist a fibre with an irreducible component $F$ of dimension $f\ge 2$. Since $-L$ is $\psi$--ample, we have
  \[   -L\vert_F \in\amp(F)\subset\bigc(F)\ .\]
Using Theorem \ref{n-1}, this implies
  \[   L\vert_F\not\in (f-1)\amp(F)\ .\]
But this leads to a contradiction: $L$ is $1$-ample, so the restriction to any subvariety must be $1$-ample as well.

We find that $\psi$ is the identity, so the MMP cannot get started and $X=X_{min}$. That is, $L$ must be nef. Since $L$ is stable, $B_+(L)=B_-(L)=\emptyset$ and $L$ is ample.  
 
\item{(\rom2)} In analogous fashion to the proof of (\rom1), it will suffice to prove:

\begin{proposition}\label{aux2} Let $X$ be as in Theorem \ref{1amp}(\rom2), and
let $L=K_X+A$, where $A$ is an ample $\RR$-divisor. Suppose $L$ is stable and
  \[ L\in 1\amp(X)\cap \bigc(X)\ .\]
Then $L$ is ample.
\end{proposition}  

To prove the proposition, consider again an $L$--MMP (which exists thanks to \cite[]{BCHM}). Let
  \[ \psi\colon X\to Z\]
be the first contraction of the program. Since $L$ is big, the contraction $\psi$ is birational. Just as above, we find that $\psi$ cannot be small, so $\psi$ must be a divisorial contraction. If all fibres of $\psi$ have dimension $\le 1$, $\psi$ is a blow--up of a smooth projective $Y$ with smooth center of codimension $2$ (Theorem \ref{dim1}); this is excluded by hypothesis. So there must be a fibre with an irreducible component $F$ of dimension $\ge 2$, which again contradicts the fact that
$L\vert_F$ is $1$-ample.

\item{(\rom3)} It will suffice to prove the following statement:

\begin{proposition}\label{aux3} Let $X$ be as in Theorem \ref{1amp}(\rom3), and let $L=K_X+A$, where $A$ is an ample $\RR$-divisor. Suppose $L$ is stable and $1$-ample. Then $L$ is nef.
\end{proposition}

We first remark that in case $L$ is big, Proposition \ref{aux3} follows from Proposition \ref{aux2}. In case $L$ is pseudo--effective, $L$ is a limit of big divisors which are stable and $1$-ample, and it follows from Proposition \ref{aux2} that $L$ is nef.
Suppose $L$ is not pseudo--effective. According to \cite[Corollary 1.3.2]{BCHM}, there exists an $L$--MMP such that on $X_{min}$ 
there is a Mori fibre space structure, i.e. a morphism 
  \[ g\colon X_{min}\to Y\]
  such that $-\phi_\ast L$ is $g$--ample. Just as in case (\rom2), we find there can be no birational contraction in the program, so we have $X=X_{min}$. If the Mori fibre space has only fibres of dimension $1$, it is a conic bundle over a smooth $Y$ (Theorem \ref{dim1}), which is excluded by hypothesis. So there exists a fibre of $g$ with an irreducible component $F$ of dimension $\ge 2$; this contradicts the fact that 
  \[L\vert_F\in 1\amp(F)\ .\]
\end{proof}

\begin{corollary}\label{kx1amp} Let $X$ be a smooth projective variety, and suppose $K_X$ is $1$-ample. Then
  \[ \vis\subset \partial\overline{\mob(X)}\ .\]
\end{corollary}

\begin{proof} It suffices to prove that 
  the relative interior $\vis^{\circ}$ is in the boundary of the mobile cone.
  But if $K_X$ is $1$-ample, every $L$ on $\vis^{\circ}$ is also $1$-ample (since $L$ is a sum of ample plus $1$-ample). But then Theorem \ref{1amp}(\rom1) implies that $L$ cannot live in the interior of $\mob(X)$.
\end{proof}

\begin{corollary}\label{convex} Let $X$ be a smooth projective variety, and suppose
  \[ \vis\cap\inte\bigl(\mob(X)\bigr)\not=\emptyset\ .\]
  Then $1\amp(X)$ is a strictly convex cone.
\end{corollary}

\begin{proof} The hypothesis implies that the dimension of $X$ is at least $3$. In case the Picard number of $X$ is $1$, the statement is clear from Theorem \ref{n-1}. Suppose the Picard number is $2$. The cone $\overline{1\amp(X)}$ has $2$ extremal rays, and by Theorem \ref{1amp}(\rom1) one of them is also an extremal ray of $\nef(X)$. On the other hand, $\overline{1\amp(X)}$ lies outside of $-\amp(X)$ (Theorem \ref{n-1}), so $\overline{1\amp(X)}$ must be convex.

The argument for Picard number $\ge 3$ is similar: in this case, we have
  \[ \dim \vis \ge 2\ ,\]
which means that $\vis$ contains infinitely many rays. Since the visible part is locally rationally polyhedral (this is the cone theorem, stated in this form in  \cite[Theorem 1]{Kaw}), there exists a ray 
  \[R\in \vis\]
 which lies in the relative interior of a face $F$ of $\nef(X)$. Let $h\subset N^1X_{\RR}$ denote the unique hyperplane containing $F$; the claim is now that $1\amp(X)$ lies on one side of $h$. To see this, suppose (by contradiction) there exists a divisor $D\in 1\amp(X)$ which lies on the ``non--ample'' side of $h$. Let $h_2\subset N^1X_{\RR}$ denote the $2$--plane spanned by $R$ and $D$. We find that any divisor $L\in R$ can be written
   \[ L=mD+A \ ,\]
  for some $m\in\RR_{>0}$ and $A$ ample (this is most easily seen by restricting attention to the $2$--plane $h_2$: by construction, $h_2$ meets $\amp(X)$, and $D$ lies outside of $-\amp(X)\cap h_2$, again by Theorem \ref{n-1}). But then $L$ is $1$--ample, contradicting Theorem \ref{1amp}(\rom1). 

\end{proof}

\begin{corollary}{("almost Fano")}\label{fano} Let $X$ be a smooth projective complex variety, and suppose that either (1) $-K_X$ is ample, or (2) $-K_X$ is $\not=0$ and nef and $\dim N^1X\ge 3$. Then:

\item{(\rom1)} 
\[\partial\nef(X)\cap\inte\bigl(\mob(X)\bigr)\] is in the boundary of $\overline{1\amp(X)}$.

\item{(\rom2)} Suppose $X$ is not the blow--up of a smooth projective variety $Y$ along a smooth codimension $2$ subvariety. Then
  \[  \partial\nef(X)\cap\bigc(X) \subset \partial\overline{1\amp(X)}\ .\]
  
\item{(\rom3)} Suppose $X$ is not a conic bundle over a smooth projective variety $Y$, nor a blow--up of a smooth projective variety along a smooth codimension $2$ subvariety. Then
  \[  \amp(X)=1\amp(X)\ .\]
\end{corollary}   

\begin{proof}

\item{(\rom1)} 
 If $-K_X$ is ample, clearly
  \[ \vis=\partial\nef(X)\ \]
  and we are done.
 Suppose now 
   \[-K_X\in\partial\nef(X)\setminus \{0\}\ .\]
Then we have
   \[  \vis=\partial\nef(X)\setminus k\ ,\]
  where $k$ denotes the ray generated by $-K_X$.  Applying Theorem \ref{1amp}(\rom1), we find an inclusion
  \[  \Bigl(\partial\nef(X)\setminus k\Bigr)\cap\inte\bigl( \mob(X)\bigr)\subset \partial\overline{1\amp(X)}\ .\]
  Suppose (\rom1) is not true, i.e. 
  \[   k\subset \inte\bigl(\mob(X)\bigr)\cap 1\amp(X)\ .\]
 Then, since $1\amp$ is an open cone, 
   \[   D:=-K_X-\epsilon A \in 1\amp(X)\]
for any ample $A$ and $\epsilon$ sufficiently small.  On the other hand, $D$ lies outside the closed cone $\nef(X)$. Let's pick an ample $\RR$-divisor $A^\prime$ close to $A$, but outside the plane spanned by $A$ and $k$ (this is possible if the ample cone has dimension $\ge 3$).
Then the line segment connecting $A^\prime$ to $D$ crosses
  \[ \Bigl(\partial\nef(X)\setminus k\Bigr)\cap \inte\bigl(\mob(X)\bigr)\ ;\]
let's call the point of intersection $B$. The $\RR$-divisor $B$ is a sum of ample and $1$-ample, hence $B$ is $1$-ample \cite[Theorem 8.3]{Tot}. On the other hand, $B$ lies in the boundary of  $\overline{1\amp(X)}$ and the $1$-ample cone is open, so $B$ cannot be $1$-ample: contradiction.

\item{(\rom2) and (\rom3)} Similar.    
\end{proof}

\begin{remark} Suppose $X$ is Fano, i.e. $-K_X$ is ample. The pseudo--index of $X$ is defined as
  \[ \tau(X)=\min \{ -K_X\cdot C\vert \ C\subset X\hbox{\ rational\ curve}\}\ .\]
If $\tau(X)$ is $\ge 2$ (respectively $\ge 3$), the hypothesis of Corollary \ref{fano}(\rom2) (respectively (\rom3)) is satisfied (this follows from Theorem \ref{wis}). In this way, we recover \cite[Proposition 29]{Lat} as a special case of Corollary \ref{fano}.
\end{remark}

\begin{corollary}{("weak Lefschetz")}\label{weak} Let $X$ be a smooth projective complex variety of dimension $n\ge 3$, and let $Y\subset X$ be a generic hyperplane section.

\item{(\rom1)} 
\[\vis\cap\inte\bigl(\mob(X)\bigr)\subset  \partial\nef(Y)\cap\partial\nef(X)\ .\]

\item{(\rom2)} Suppose $X$ is not the blow--up of a smooth projective variety $Y$ along a codimension $2$ smooth subvariety. Then
  \[  \vis\cap\bigc(X) \subset \partial\nef(Y)\cap\partial\nef(X)\ .\]
  
\item{(\rom3)} Suppose $X$ is not a conic bundle over a smooth projective variety, nor a blow--up of a smooth projective variety along a smooth codimension $2$ subvariety. Then
  \[  \vis \subset \partial\nef(Y)\cap\partial\nef(X)\ .\]
\end{corollary}

The following is an alternative formulation of Corollary \ref{weak}(\rom1). The reformulation of points (\rom2) and (\rom3) is left to the diligent reader.

\begin{corollary}{("ampleness criterion")} Let $X$ be a smooth projective variety of dimension $n\ge 3$, and let $L$ on $X$ be a divisor of the form $L=K_X+A$, with $A$ an ample $\RR$-divisor. Suppose $L\in\inte\bigl(\mob(X)\bigr)$. Then $L$ is ample if and only if $L\vert_Y$ is ample for some generic hyperplane $Y\subset X$.
\end{corollary}

Combining Corollaries \ref{fano} and \ref{weak}, we get in particular:

\begin{corollary}{("weak Lefschetz for almost Fano")}\label{weakfano} Let $X$ be a smooth projective complex variety of dimension $n\ge 3$. Suppose either (1) $-K_X$ is ample, or (2) $-K_X$ is nef and $\not=0$ and $\dim N^1X\ge 3$. Let $Y\subset X$ be a very ample divisor, generic in its linear system.

\item{(\rom1)} 
  \[ \partial\nef(X)\cap\inte\bigl( \mob(X)\bigr)\subset \partial\nef(Y)\cap\partial\nef(X)\ .\]
  
\item{(\rom2)} Suppose $X$ is not the blow--up of a smooth variety along a smooth codimension $2$ subvariety.  Then
  \[ \partial\nef(X)\cap \bigc(X)\subset \partial\nef(Y)\cap\partial\nef(X)\ .\]
  
\item{(\rom3)} Suppose $X$ is not a conic bundle over a smooth projective variety, nor a blow--up of a smooth projective variety along a smooth codimension $2$ subvariety. Then
  \[   \partial\nef(X)\subset \partial\nef(Y)\ .\]
  
\item{(\rom4)} Let $X$ be as in (\rom3) and $n\ge 4$. Then restriction induces an isomorphism
  \[   \nef(X)\cong\nef(Y)\ .\]  
\end{corollary}

\begin{remark} The statement of Corollary \ref{weakfano}(\rom4) for $X$ Fano was originally proven by 
Wi\'sniewski \cite[p. 147 Corollary]{Wis}. This provided the starting--block for much further work concerning weak Lefschetz for the ample cone (\cite{Has}, \cite{Jow}, \cite{ANO}, \cite{AO}, \cite{BI}, \cite{Sz}).
\end{remark}

\section{$q$-ample}

This section is about the cone of $q$-ample divisors. We prove the result stated in the introduction:

\begin{theorem} \label{qample} Let $X$ be a smooth projective variety of dimension $n$.  For any non--negative integer $q$, we have

\[   \vis\cap B(n-1-q)\amp(X)\subset \partial\overline{q\amp(X)}\ .\]
\end{theorem}

We actually prove a more general statement:

\begin{theorem}\label{qample2} Let $X$ be a smooth projective variety of dimension $n$, and define
  \[ \tau=\min\{ \ell(R)\vert \hbox{\ $R$\ is\ a\ $K_X$-negative\ extremal\ ray}\}\ .\]

\item{(\rom1)} For any non--negative integer $q$ such that $q\ge\tau-2$, we have

\[   \vis\cap B(n+\tau-q-2)\amp(X)\subset \partial\overline{q\amp(X)}\ .\]

\item{(\rom2)} Suppose $X$ is not the blow--up of a smooth variety $Y$ along a smooth subvariety of codimension $\ge 2$.
Then
  \[   \vis\cap \bigc(X)\subset \partial\overline{(\tau)\amp(X)}\ .\]
 \end{theorem}
 
 \begin{proof}
 
 \item{(\rom1)} As in the proof of Theorem \ref{1amp}, one can restrict attention to the relative interior $\vis^{\circ}$ and hence it suffices to prove the following:
 
 \begin{proposition}\label{auxi} Let $L$ be a divisor of the form $L=K_X+A$, with $A$ an ample $\RR$-divisor. Suppose $L$ is stable and
   \[  L\in B(n+\tau-q-2)\amp(X)\cap q\amp(X)\ .\]
 Then $L$ is ample.
 \end{proposition}
 
To prove the proposition, consider an $L$--MMP
  \[\phi\colon  X=X_0- \to X_1- \to  \cdots - \to X_{min}\ ,\]
  where either $\phi_\ast L$ is semi--ample on $X_{min}$ (if $L$ is big), or there exists a Mori fibre space structure on $X_{min}$ (if $L$ is not pseudo--effective). This exists thanks to \cite{BCHM}.
Let 
  \[ \psi\colon X\to Z\]
  denote the first contraction of the $L$--MMP, let $V\subset X$ denote the exceptional locus of $\psi$ and let $F$ be a general fibre of $\psi\vert_V$.  Note that $K_X<L$ so that $\psi$ corresponds to the contraction of a $K_X$--negative extremal ray and 
  Wi\'sniewski's theorem (Theorem \ref{wis}) applies. This gives
  \[ \dim V +\dim F\ge n+\tau-1\ .\]
Since $V\subset B_+(L)$ (Lemma \ref{indet}), its dimension is $\le n+\tau-q-2$. It follows that
  \[ \dim F\ge q+1\ .\]
By construction, $-L$ is $\psi$--ample, hence
  \[ L\vert_F \in -\amp(F)\ .\]
On the other hand the restriction of $L$ to any subvariety is $q$--ample, so in particular
  \[ L\vert_F\in q\amp(F)\ .\]
 But this is not possible if $\dim F\ge q+1$:
   \[ q\amp(F)\subset (\dim F -1)\amp(F)\ ,\]
  and the cone $(\dim F-1)\amp(F)$ is the complement of $-\psef(F)$: contradiction.
  
 Since the $L$--MMP cannot get started, it is trivial. That is, either $L$ is nef on $X$, or there exists a contraction of fibre type
   \[ g\colon X\to Z\]
which is $L$--negative and $K_X$--negative. The second possibility can be excluded, again using Wi\'sniewski's theorem: if $F$ is a general fibre of $f$, 
we have
  \[ n+\dim F   \ge n+\tau-1\ ,\]
i.e. there is a fibre $F$ of dimension $\ge \tau-1$. But supposing there is a fibre type contraction, $L$ is not big which is only possible if $q=\tau-2$. So $L$ and the restriction $L\vert_F$ are $(\tau-2)$--ample, which contradicts the fact that
  \[ L\vert_F\in -\amp(F)\subset -\bigc(F)\ .\]

 \item{(\rom2)} This follows once we have proven the following:
 
 \begin{proposition}\label{auxi2} Let $X$ be as in Theorem \ref{qample2}(\rom2), and let $L$ be a divisor of the form $L=K_X+A$, with $A$ an ample $\RR$-divisor. Suppose $L$ is big and $\tau$-ample. Then $L$ is ample.
 \end{proposition}
 
 To prove the proposition, we apply \cite[]{BCHM} to get an $L$--MMP
   \[\phi\colon  X=X_0- \to X_1- \to  \cdots - \to X_{min}\ ,\]
   where $\phi_\ast L$ on $X_{min}$ is nef. Consider the first contraction
   \[\psi\colon X\to Z\]
   in this $L$--MMP. As above, let $V\subset X$ denote the exceptional locus of $\psi$ and let $F$ be a general fibre of $\psi\vert_V$.  Note that $K_X<L$ so that $\psi$ corresponds to the contraction of a $K_X$--negative extremal ray and hence
  Wi\'sniewski's theorem (Theorem \ref{wis}) applies to $\psi$.
  If $\psi$ is a small contraction (i.e. $\dim V\le n-2$), Wi\'sniewski's theorem gives
    \[   \dim F \ge \tau+1\ ,\]
   and we get a contradiction with the fact that $L\vert_F$ is $\tau$-ample. So $\psi$ must be a divisorial contraction, and all fibres of $\psi\vert_V$ must be of dimension equal to $\tau$ (by Wi\'sniewski's theorem, each fibre has dimension $\ge \tau$, while the fact that $L$ is $\tau$-ample implies that each fibre has dimension $\le \tau$). In this case, a result of  
 Andreatta--Occhetta \cite[Theorem 5.1]{AO2} informs us that $\psi$ identifies $X$ with a blow--up of
some smooth projective variety $Y$ along a smooth subvariety; this is excluded by hypothesis. 

Altogether, we find there can be no contraction and hence $X=X_{min}$ and $L$ is already nef. It remains to prove ampleness of $L$. To this end, note that
  \[  L^\prime=K_X+(1-\epsilon)A\]
  is still big and $(\tau-1)$-ample for $\epsilon$ sufficiently small (since $\bigc(X)$ and $(\tau-1)\amp(X)$ are open cones). Applying the above reasoning to $L^\prime$, we find that $L^\prime$ is nef. But then
  \[  L=L^\prime + \epsilon A\]
  is ample.
\end{proof}

 \begin{corollary}{("weak Lefschetz")} Let $X$ and $\tau$ be as in Theorem \ref{qample2}. 
  
 \item{(\rom1)} Let $Y\subset X$ be a generic complete intersection of codimension $q\le n-2$. Then
 \[   \vis\cap B(n+\tau-q-2)\amp(X)\subset  \partial\nef(Y)\cap\partial\nef(X)\ .\]
 
 \item{(\rom2)} Suppose $X$ is not the blow--up of a smooth variety along a smooth subvariety of codimension $\ge 2$. Let $Y\subset X$ be a generic complete intersection of codimension $\tau$. Then
   \[   \vis\cap \bigc(X)\subset   \partial\nef(Y)\cap\partial\nef(X)\ .\]
\end{corollary}

\begin{proof} This is immediate from Theorem \ref{qample}, once one knows that $H q$-ample implies $q$-ample (Proposition \ref{inclusions}).
\end{proof}


\begin{thebibliography}{dlPG99}


\bibitem{ANO} M. Andreatta, C. Novelli and G. Occhetta, Connections between the geometry of a projective variety and of an ample section, Math. Nachr. 279 (2006), 1387---1395,

\bibitem{AO} M. Andreatta and G. Occhetta, Extending extremal contractions from an ample section, Adv. Geom. 2 (2002), 133---146,

\bibitem{AO2} M. Andreatta and G. Occhetta, Special rays in the Mori cone of a projective variety, Nagoya Math. J. Vol. 168 (2002), 127---137,

\bibitem{AW} M. Andreatta and J. Wi\'sniewski, A view on contractions of higher dimensional varieties, in: Algebraic geometry (Santa Cruz 1995), Proc. Symp. Pure Math. 62, Amer. Math. Soc., Providence, 1997,

\bibitem{BCHM} C. Birkar, P. Cascini, C. Hacon and J. McKernan, Existence of minimal models
for varieties of log general type, J. Amer. Math. Soc. 23 (2010), 405---468,

\bibitem{BI} M. C. Beltrametti and P. Ionescu, A view on extending morphisms from ample divisors, in: Interactions of Classical and Numerical Algebraic Geometry (D. J. Bates et alii, eds.), American Math. Society,

\bibitem{BBP} S. Boucksom, A. Broustet and G. Pacienza, Uniruledness of stable base loci of adjoint linear systems via Mori theory, Math. Z. 275 (2013), 499---507,

\bibitem{Choi} S. Choi, Duality of the cones of divisors and curves, Math. Res. Lett. 19 (2012), 403---416,

\bibitem{Choi2} S. Choi, On the dual of the mobile cone, Math. Z. 272 (2012), 87---100,

\bibitem{DPS} J.-P. Demailly, T. Peternell and M. Schneider, Holomorphic line bundles with partially vanishing cohomology, Proceedings of the Hirzebruch 65 conference on
algebraic geometry (Ramat Gan, 1993), Bar-Ilan Univ. (1996),

\bibitem{ELM} L. Ein, R. Lazarsfeld, M. Mustat\v{a}, M. Nakamaye and M. Popa, Restricted volumes
and base loci of linear series, Amer. J. Math. 131 (2009), no. 3, 607---651,

\bibitem{ELM2} L. Ein, R. Lazarsfeld, M. Mustat\v{a}, M. Nakamaye and M. Popa, Asymptotic invariants
of base loci, Annales de l'Institut Fourier, 56, (2006) no.6, 1701---1734,

\bibitem{FKL} T. de Fernex, A. K\"uronya and R. Lazarsfeld, Higher cohomology of divisors on a projective variety, Math. Ann. 337 No 2 (2007), 443---455,

\bibitem{Ful} W. Fulton, Intersection theory, Springer--Verlag, Berlin Heidelberg New York 1984,


\bibitem{Has} B. Hassett, H.-W. Lin and C.-L. Wang, The weak Lefschetz principle is false for ample cones, Asian Journal of Mathematics 6 (2002), No. 1, 95---100,

\bibitem{HK} Y. Hu and S. Keel, Mori dream spaces and GIT, Michigan Math. J. 48 (2000),
331---348,

\bibitem{Ion} P. Ionescu, Generalized adjunction and applications, Math. Proc. Cambridge Philos. Soc.
99 (1986), 457---472,

\bibitem{Jow} S.-Y. Jow, A Lefschetz hyperplane theorem for Mori dream spaces, to appear in Math. Z.,


\bibitem{Kaw} Y. Kawamata, Remarks on the cone of divisors, in: Classification of algebraic varieties (C. Faber et alii, eds.), European Math. Society, Z\"urich 2011,


\bibitem{Kur} A. K\"uronya, Positivity on subvarieties and vanishing theorems for higher cohomology, Annales de l?Institut Fourier 63 (2013),

\bibitem{Lat} R. Laterveer, The weak Lefschetz principle and cones of divisors, preprint,

\bibitem{Pay} S. Payne, Stable base loci, movable curves, and small modifications, for toric
varieties, Math. Z. 253 (2006), no. 2, 421---431,

\bibitem{Som} A. Sommese, Submanifolds of abelian varieties, Math. Ann. 233 (1978), 229---256,

\bibitem{Sz} B. Szendr\"oi, On the ample cone of an ample hypersurface, Asian Journal of Mathematics 7 (2003), 1---6,

\bibitem{Tot} B. Totaro, Line bundles with partially vanishing cohomology, J. Eur. Math. Soc. 15 (2013), 731---754,

\bibitem{Wis} J. Wi\'sniewski, On contractions of extremal rays of Fano manifolds, J. reine angew. Math. 417 (1991), 141---157,

\end{thebibliography}
\end{document}